\documentclass[12pt]{article}

\usepackage[text={17cm,24cm}]{geometry}

\usepackage{bm}
\usepackage{amsfonts}
\usepackage{graphicx}
\usepackage{epstopdf}
\usepackage{algorithmic}
\usepackage{mathtools}

\usepackage{tikz}

\ifpdf
  \DeclareGraphicsExtensions{.eps,.pdf,.png,.jpg}
\else
  \DeclareGraphicsExtensions{.eps}
\fi

\usepackage{amsmath,amssymb,amsthm,amscd,amsfonts}

\usepackage{euscript}
\usepackage{graphicx}

\usepackage[colorlinks, urlcolor=blue, pdfborder={0 0 0 [0 0]}]{hyperref}

\DeclareMathOperator{\tr}{tr}
\DeclareMathOperator{\diag}{diag}

\def\Var{\mathsf{D}}

\def\rank{\mathop{\mathrm{rank}}}

\def\diag{\mathop{\mathrm{diag}}}

\DeclareSymbolFont{bbold}{U}{bbold}{m}{n}
\DeclareSymbolFontAlphabet{\mathbbold}{bbold}

\newcommand\eqdef{\stackrel{\mathclap{\textrm{\mbox{def}}}}{=}}

\usepackage[english]{babel}
\newtheorem{remark}{Remark}
\newtheorem{corollary}{Corollary}
\newtheorem{proposition}{Proposition}

\newtheorem{theorem}{Theorem}
\newtheorem{lemma}{Lemma}

\usepackage{euscript}

\newcommand{\cH}{\EuScript{H}}

\newcommand{\cM}{\EuScript{M}}

\newcommand{\bfA}{\mathbf{A}}
\newcommand{\bfB}{\mathbf{B}}
\newcommand{\bfC}{\mathbf{C}}

\newcommand{\bfI}{\mathbf{I}}

\newcommand{\bfL}{\mathbf{L}}

\newcommand{\bfR}{\mathbf{R}}

\newcommand{\bfU}{\mathbf{U}}

\newcommand{\bfW}{\mathbf{W}}
\newcommand{\bfX}{\mathbf{X}}
\newcommand{\bfY}{\mathbf{Y}}

\newcommand{\calT}{\mathcal{T}}

\newcommand{\spC}{\mathbb{C}}

\newcommand{\spR}{\mathbb{R}}

\newcommand{\tsC}{\mathsf{C}}
\newcommand{\tsD}{\mathsf{D}}

\newcommand{\tsP}{\mathsf{P}}
\newcommand{\tsQ}{\mathsf{Q}}

\newcommand{\tsU}{\mathsf{U}}
\newcommand{\tsV}{\mathsf{V}}

\newcommand{\bfa}{\mathbf{a}}
\newcommand{\bfb}{\mathbf{b}}
\newcommand{\bfc}{\mathbf{c}}

\newcommand{\bfs}{\mathbf{s}}

\newcommand{\bfu}{\mathbf{u}}

\newcommand{\bfx}{\mathbf{x}}
\newcommand{\bfy}{\mathbf{y}}
\newcommand{\bfz}{\mathbf{z}}

\newcommand{\tsa}{\mathsf{a}}
\newcommand{\tsb}{\mathsf{b}}
\newcommand{\tsc}{\mathsf{c}}

\newcommand{\tsp}{\mathsf{p}}
\newcommand{\tsq}{\mathsf{q}}
\newcommand{\tsr}{\mathsf{r}}

\newcommand{\tsu}{\mathsf{u}}
\newcommand{\tsv}{\mathsf{v}}
\newcommand{\tsw}{\mathsf{w}}
\newcommand{\tsx}{\mathsf{x}}
\newcommand{\tsy}{\mathsf{y}}
\newcommand{\tsz}{\mathsf{z}}
\renewcommand{\bfL}{\bfA}
\renewcommand{\bfR}{\bfB}

\newcommand{\bt}{\begin{theorem}}
\newcommand{\et}{\end{theorem}}
\newcommand{\bl}{\begin{lemma}}
\newcommand{\el}{\end{lemma}}
\newcommand{\bp}{\begin{proposition}}
\newcommand{\ep}{\end{proposition}}
\newcommand{\bc}{\begin{corollary}}
\newcommand{\ec}{\end{corollary}}

\newcommand{\bd}{\begin{definition}\rm}
\newcommand{\ed}{\end{definition}}
\newcommand{\bex}{\begin{example}\rm}
\newcommand{\eex}{\end{example}}
\newcommand{\br}{\begin{remark}\rm}
\newcommand{\er}{\end{remark}}

\newcommand{\btbh}{\begin{table}[!ht]}
\newcommand{\etb}{\end{table}}
\newcommand{\bfgh}{\begin{figure}[!ht]}
\newcommand{\efg}{\end{figure}}

\newcommand{\bea}{\begin{eqnarray*}}
\newcommand{\eea}{\end{eqnarray*}}
\newcommand{\be}{\begin{eqnarray}}
\newcommand{\ee}{\end{eqnarray}}
\newcommand\Expect{\mathsf{E}}

\def\rank{\mathop{\mathrm{rank}}}

\def\tr{\mathop{\mathrm{tr}}}

\makeatletter
\def\adots{\mathinner{\mkern2mu\raise\p@\hbox{.}
\mkern2mu\raise4\p@\hbox{.}\mkern1mu
\raise7\p@\vbox{\kern7\p@\hbox{.}}\mkern1mu}}
\newcommand{\l@abcd}[2]{\hbox to\textwidth{#1\dotfill #2}}
\makeatother

\begin{document}

\author{Nina Golyandina\footnote{St.Petersburg State University, Universitetskaya nab. 7/9, St.Petersburg, 199034, Russia; n.golyandina@spbu.ru; the work is supported by RFBR, project number 20-01-00067},
Anatoly Zhigljavsky\footnote{Cardiff University, School of Mathematics,  Cardiff CF24 4AG, UK; zhigljavskyaa@cardiff.ac.uk}}
\title{Blind deconvolution of covariance matrix inverses for autoregressive processes}
\date{}
\maketitle

\begin{abstract}
Matrix $\bfC$ can be blindly deconvoluted if there exist matrices $\bfA$ and $\bfB$ such that $\bfC= \bfA \ast \bfB$, where
$\ast$ denotes the operation of matrix convolution. We study the problem of matrix deconvolution in the case where
 matrix $\bfC$ is proportional to the inverse of the autocovariance matrix of an autoregressive process.
 We show that the deconvolution of such matrices is important in problems of Hankel structured low-rank approximation (HSLRA).
In the cases of  autoregressive models of orders one and two, we fully characterize the range of parameters  where such deconvolution  can be performed
and provide construction schemes for performing deconvolutions. We also consider general autoregressive models of order $p$, where we prove that the
deconvolution $\bfC= \bfA \ast \bfB$ does not exist if the matrix $\bfB$ is diagonal and its size is larger than $p$.
\end{abstract}



\section{Introduction}

Let $A$ and $B$  be positive integers  and
$\bfa = (a_0, \ldots, a_{A})^\top \in \spR^{A+1}$ and $\bfb = (b_0, \ldots, b_{B})^\top\in \spR^{B+1}$ be two vectors.
The convolution of vectors $\bfa$ and $\bfb$ is the  vector $\bfc = \bfa \ast \bfb=(c_0, \ldots, c_{C})^\top \in \spR^{C+1} $
with $C = A + B $ and  $c_i = \sum_k a_k b_{i-k},$ where $i=0,1, \ldots, C\, $ and the sum taken over the set of indices $k$  such that the elements $a_{k}$ and $b_{i-k}$ are defined (that is, $ \max\{0,i-B\} \leq k \leq  \min\{A,i\}$).

The definition of vector convolution naturally extends to matrices as follows.
Let $\bfA=(a_{i,j})_{i,j=0}^A$ and $\bfB=(b_{i, j})_{i,j=0}^B$  be matrices of sizes $(A\!+\!1)\!\times\!(A\!+\!1)$ and
$(B\!+\!1)\!\times\!(B\!+\!1)$, respectively.
A matrix $\bfC = (c_{i, j})_{i,j=0}^{A+B}$  is a convolution
of the matrices $\bfA$
and $\bfB$ if
$ c_{i,j}= \sum_{k,l} a_{k,l}b_{i-k,j-l}$, where the sum is taken over the sets of indices $k$ and $l$ such that the elements $a_{k,l}$ and $b_{i-k,j-l}$ are defined; that is, $ \max\{0,i-B\} \leq k \leq \min\{A,i\}$ and $\max\{0,j-B\} \leq l \leq  \min\{A,j\}$.

The generating function (gf) of a vector $\bfu=(u_0,\ldots,u_M)^\top$ is defined as
$G_\bfu(t) = u_0 +u_1 t + \ldots + u_M t^M\,$. Similarly, the generating function of a matrix $\bfU=(u_{i,j})_{i,j=0}^M$
is $G_\bfU(t,s)= \sum_{i,j=0}^M u_{i,j} t^i s^j  $, where $t,s \in \spC$.

From the definition of  convolution, $\bfc = \bfa \ast \bfb$ if and only if $G_\bfc(t)= G_\bfa(t) G_\bfb(t)$ for all $t \in \spC$.
Similarly,
$\bfC= \bfA \ast \bfB$ if and only if $G_\bfC(t,s)= G_\bfA(t,s) G_\bfB(t,s)$ for all $t,s \in \spC$. This yields that all statements about vector and matrix convolutions and deconvolutions can be equivalently formulated in the language of generating functions.

A vector $\bfc$ can be blindly deconvoluted if there exist vectors $\bfa$ and $\bfb$ such that $\bfc = \bfa \ast \bfb$.
This problem (interesting only under some restrictions on $\bfa$ and $\bfb$)  is equivalent to studying the roots of the gf $G_\bfc(t)$.

A matrix $\bfC $ can be blindly deconvoluted if there exist matrices $\bfA$ and $\bfB$ such that $\bfC= \bfA \ast \bfB$.

There exists an extensive literature related to matrix convolution and (blind) deconvolution but it is mostly related to applications in image processing; see e.g. \cite{Kundur.Hatzinakos1996, Michailovich.Tannenbaum2007}.
We are interested in the blind deconvolution of matrices $\bfC$ which are proportional to inverses of the autocovariance matrices of autoregressive processes (ARs). As argued in the next section, this problem has significant practical importance in  signal processing and time series analysis.

The structure of the paper is as follows. In Section 2 we discuss the practical importance of  the stated problem and establish that  two different types of norms in the HSLRA problem are equivalent
if and only if the matrix, inverse to the autocovariance matrix of the noise process, can be blindly deconvoluted.
In Section 3 we prove  auxiliary statements about generating functions (gf) of diagonals of matrices and relate gf of banded matrices to gf of its diagonals. In Section 4 we provide the main matrices of interest (inverses to autocovariance matrices for first-order AR(1) and second-order AR(2) autoregressive processes) and derive important relations between gf of diagonals of these matrices. In Section 5 we prove our main results establishing ranges of parameters in the AR(1) and AR(2) models, where the deconvolution can be performed.
In particular, we  establish that  for the stationary AR(1) and AR(2) models the
inverse to the autocovariance matrix cannot be deconvoluted. In cases where the deconvolution is possible, we  provide the construction schemes for the matrices $\bfA$ and $\bfB$. In Section 6 we specialize  results of Section 5 when we require an additional condition of non-negative definiteness for  the matrices
$\bfA$ and $\bfB$.
In Section 7 we derive some partial results for the  stationary AR($p$) model with general $p\geq 1$. Section 8 concludes the paper.

\section{Motivation: selection of a matrix norm in HSLRA}

Our  {motivation} is  the problem of extraction of a signal $\bfs=(s_0,s_1, \ldots, s_{N-1})^\top$
from an observed noisy signal $\bfx=(x_0,x_1, \ldots, x_{N-1})^\top=\bfs+\bm\xi$ of length $N$,
where $\bm\xi=(\xi_0,\xi_1, \ldots, \xi_{N-1})^\top$ is a vector of (unobserved) random noise
with zero mean and covariance matrix $\bm\Sigma = \Expect \bm\xi \bm\xi^\top$.

We consider a wide class of signals, which have an explicit parametric form of a finite sum:
$$
  s_n = \sum_k \tsp_k(n)\exp(\alpha_k n) \sin(2\pi \omega_k n + \phi_k),
$$
where $\alpha_k, \omega_k$ and $\phi_k$ are arbitrary real numbers and $\tsp_k(n)$ are polynomials in $n$.

This class of signals can be defined through  low-rank Hankel matrices as follows.
Set a window length $L$, $1<L\leq N/2$; $K = N - L +1$.
With a series $\bfz=(z_0,z_1, \ldots, z_{N-1})^\top$, we associate  the so-called trajectory matrix
\begin{eqnarray*}
  \calT_L(\bfz) = \left(
  \begin{array}{llll}
    z_0      & z_1 & \dots     & z_{K-1}     \\
    z_1      & \cdots  & \;\cdots    & z_{K}    \\
    \; \ldots   & \;\ldots   & \ \:\ldots    & \; \ldots   \\
    z_{L-1}  & z_{L} & \dots & z_{N-1}
  \end{array}
  \right) \in \spR^{L \times K}.
\end{eqnarray*}
If there exists an integer $r<\lfloor N/2 \rfloor $ such that $r$ is the rank of $\calT_L(\bfz)$ for any $L\in [r,  \lfloor N/2 \rfloor ]$, then we  say that  $\bfz$ is a \emph{series of rank $r$} {and write $\rank \bfz = r$}.

The problem of finite-rank signal extraction can be reduced to the problem of approximation of the $L$-trajectory matrix $\calT_L(\bfx)$ of
the observed time series $\bfx$ by a Hankel matrix of rank $r$. This problem belongs to
the class of problems of Hankel structured low-rank approximation (HSLRA), see e.g. \cite{Chu.etal2003,Markovsky2019,Markovsky.etal2006}.

The HSLRA problem can be stated in two forms: (a) vector form and (b) matrix form.
The vector (time series) form of this problem is: for given $\bfx \in \spR^{N}$ and positive integer $r<\lfloor N/2 \rfloor $,
\begin{equation}\label{eq:L-rank_task}
\|\bfx-\bfy\|^2_{\bfW} \to \min_{\bfy: \rank \bfy \le r}\,,
\end{equation}
where $\bfy = (y_0, y_1, \ldots, y_{N-1})^\top$, $\|\bfz\|_{\bfW}^2 = \bfz^\top \bfW \bfz$ for $\bfz \in \spR^{N}$ and $\bfW$ is some positive definite matrix of size $N\times N$.

The solution of \eqref{eq:L-rank_task} can be considered as a weighted least-square estimate (WLSE) of the signal $\bfs$.
If noise $\bm\xi$ is Gaussian with covariance matrix $\bm\Sigma$, then the WLSE with $\bfW_0 = \bm\Sigma^{-1}$ is the maximum likelihood estimate (MLE).
If the properties of the noise process are known then the vector form \eqref{eq:L-rank_task}
is the most natural way of defining the HSLRA problem. However, solving the HSLRA problem in the vector form is extremely difficult, see e.g. \cite{gillard2013optimization}.
{Although the vector form allows fast implementations, these implementations are very complex
and need a starting point that is close to the solution \cite{Markovsky.Usevich2014,Zvonarev.Golyandina2018}.}

The matrix form of the HSLRA problem  allows one to use simple subspace-based alternating projection methods (e.g., the Cadzow iterations \cite{Cadzow1988})
and hence is computationally much   preferable
than the vector form \eqref{eq:L-rank_task};  see \cite[Sect. 3.4]{golyandina2018singular} for details.
Note also that the well-known method called singular spectrum analysis (SSA) can formally be considered as one Cadzow  iteration
and therefore  it is also a subspace-based method and thereby related to the matrix form of the HSLRA; see \cite{Golyandina} for a modern introduction to the methodology of SSA and  \cite{golyandina2018singular} for a comprehensive overview of SSA.

Define the inner product in $\spR^{L \times K}$ as $\langle \bfX, \bfY \rangle = \langle \bfX, \bfY \rangle_{\bfL, \bfR} = \tr (\bfL \bfX \bfR \bfY^\top)$, where $\bfL =(a_{i,j})_{i,j=0}^{L-1} \in \spR^{L \times L}$, $\bfR=(b_{i,j})_{i,j=0}^{K-1}  \in \spR^{K \times K}$; $\| \bfX \|_{\bfL, \bfR}$ is the corresponding matrix norm in $ \spR^{L \times K}$.
The  HSLRA  problem in the matrix form is the following optimization problem:
\begin{equation}\label{eq:rank_task}
  \|\bfX - \bfY \|_{\bfL, \bfR}^2 \to \min_{\bfY\in \cM_r \cap \cH}\;, \;\;
\end{equation}
where
$\cH \subset \spR^{L\times K}$ is the space of Hankel matrices of size $L \times K$, $\cM_r\subset \spR^{L\times K}$ is the set of matrices of rank not larger than $r$.
For reformulating  the original HSLRA problem  \eqref{eq:L-rank_task}  in the matrix form \eqref{eq:rank_task}, we have to choose   $\bfX = \calT_L(\bfx)$ and $\bfY = \calT_L(\bfy)$; the remaining issue then is to match the vector norm in \eqref{eq:L-rank_task}  with the  matrix norm in \eqref{eq:rank_task}.

Particular cases of the correspondence between the vector-norm and matrix-norm formulations \eqref{eq:L-rank_task} and \eqref{eq:rank_task} of the HSLRA problem are considered in \cite{gillard2015stochastic,Zvonarev.Golyandina2017}. The general case
is established in the following theorem.

\begin{theorem}\label{th:W_L_R}
For any  $\bfz \in \spR^{N}$,
$\|\calT_L(\bfz)\|_{\bfL, \bfR} = \|\bfz\|_\bfW$  if and only if
\be
\label{eq:b_decomp}
\bfW = \bfL \ast \bfR \, .
\ee
\end{theorem}

\begin{proof}
Consider the squared norm $\|\mathbf{X}\|_{\bfL, \bfR}^2={\tr (\mathbf{A} \mathbf{X}\mathbf{B}\mathbf{X}^\top})$ with $\mathbf{X}= \calT_L(\bfz)$ so that $x_{l,k}=z_{l+k}$ for $l=0, \ldots, L-1$ and $k=0, \ldots,K-1$.
We have
\begin{eqnarray*}
  \|\mathbf{X}\|_{\mathbf{A,B}}^2=
  \sum_{l,l',k,k'}  a_{l,l'} x_{l',k'} b_{k',k} x_{l,k}
  =\sum_{l,l',k,k'}    a_{l,l'} z_{l'+k'} b_{k',k} z_{l+k}
  = \sum_{l,l',k,k'}   z_{ n'} a_{l,l'}  b_{ n'-l',{n-l}} z_{n}
\end{eqnarray*}
where $n=k+l$, $n'=k'+l'$ and all sums above are taken for $l,l'=0, \ldots, L-1$ and $k,k'=0, \ldots,K-1$. By changing the summation indices in the last sum  $k \to n$ and $k' \to n'$ we obtain the required.
\end{proof}

In a typical application, when the structure of the noise in the model `signal plus noise' is assumed, the HSLRA problem is formulated in a vector form with a given matrix $\bfW$.
As mentioned above, algorithms of solving the HSLRA problem are much easier if we have the matrix rather than vector form of  the HSLRA problem.
Therefore,  in view of Theorem~\ref{th:W_L_R}, for a given $\bfW$ we would want to find positive definite matrices $\bfA$ and $\bfB$ such that
\eqref{eq:b_decomp} holds; that is, we would want to perform  a blind deconvolution of the matrix $\bfW$.

Matrices $\bfA$ and $\bfB$ in \eqref{eq:rank_task} and therefore in \eqref{eq:b_decomp}  should be symmetric non-negative definite, see e.g. \cite{Allen2014} and \cite[p.62]{golyandina2018singular}. In Theorems 2--4 below we shall require symmetry of matrices  $\bfA$ and $\bfB$ in \eqref{eq:b_decomp} and in Section~\ref{sec:ABnonnegative} we discuss whether they can be chosen to be non-negative definite.

It follows from the results of \cite{Zhigljavsky.etal2016} that in the case when the noise $\bm\xi$ is white, and therefore $\bfW = \bfI_N$, the matrix $\bfW$ cannot be blindly deconvoluted under the condition that $\bfA$ and $\bfB$ are positive definite matrices; however, for a wide range of parameters $N$ and $L$ there are many pairs of non-negative definite diagonal matrices $\bfA$ and $\bfB$ such that \eqref{eq:b_decomp} holds.
This paper extends  results of \cite{Zhigljavsky.etal2016} to the  case of banded matrices corresponding to the case where the noise~$\bm\xi$ forms an autoregressive process.

The  white-noise model is  the simplest and hence the most popular model of noise  used for formulation of the `signal plus noise' problems.
The autoregressive model of noise  is the second most common noise model used in such problems.
In particular, in climatology,  the most common  model of noise is the so-called `red noise'; that is, an auto-regressive process of order one with a positive coefficient.
Red noise  suits SSA and related methods very well, since the spectral density of red noise is monotonic.
This was the principal  reason for the creation in \cite{Allen.Smith1996} of the method called  `Monte Carlo SSA'. This method, where the assumption of red noise is crucial,  has been used and further developed in many papers including \cite{Allen.Robertson1996}, \cite{Palus.Novotna1998}, \cite{PALUS.Novotna2004} and \cite{Jemwa.Aldrich2006}. Monte Carlo SSA serves for detection of signals in red noise and is currently used for analysing time series in different areas, most notably climatology and geophysics; for example, for  investigation of ice conditions \cite{Jevrejeva.Moore2001}, sea surface temperature dynamics \cite{Allen.Smith1996}, and GPS observations \cite{Xu.Yue2015}.

In Monte Carlo SSA, the ordinary singular-value decomposition (SVD) is used as the first step for obtaining the basis of the signal subspace.
The use of the oblique SVD with AR-generated weights in Monte Carlo SSA could extend the applicability of the method.
Therefore, in addition to  understanding of the equivalence between the vector and matrix forms of the HSLRA, a theoretical investigation of  the HSLRA, SSA  and other subspace-based methods with the  autoregressive noise  seems to be   important too.

\section{Generating functions and convolution of banded matrices}

\subsection{Generating function of a matrix via generating functions of its diagonals}

In some cases (e.g, for banded matrices), it is natural to construct  generating functions of matrices as
a sum of generating functions of diagonals.
Consider a matrix  $\bfU=\{u_{i,j}\}_{i,j=0}^M$ and let $ G_\bfU(t,s)$ be its generating function (gf).

Define $\diag_i({\bfU})$, the $i$-th diagonal of  $\bfU$, as the vector of length $M-|i|+1$ with elements with indices $(j,k)$ satisfying $k-j=i$, $i=-M,\ldots,M$.
This $i$-th diagonal $\diag_i({\bfU})$ has the univariate gf
$\tsu_i (\tau)= G_{\diag_i({\bfU})}(\tau) = \sum_{j=0}^{M-|i|} u_{j,j+i} \tau^j$.

\begin{lemma}\label{lem:gf2}
\bea
  G_\bfU(t,s) = \sum_{i=-M}^{M} t^{(|i|-i)/2} s^{(|i|+i)/2} \tsu_{i} (ts).\\
\eea
If $\bfU$ is symmetric,  this formula simplifies to
$
  G_\bfU(t,s) = \tsu_{0} (ts) + \sum_{i=1}^{M} (t^i + s^i) \tsu_{i} (ts).
$
\end{lemma}

\begin{proof}
\begin{eqnarray*}
  G_\bfU(t,s) &=& \sum_{j,k=0}^M u_{j,k} t^j s^k = \sum_{i=-M}^M \sum_{k-j=i} u_{j,k} t^j s^k =\\
  &=& \sum_{i=1}^M t^i \sum_{j=0}^{M-i} u_{j+i,j} t^j s^j + \sum_{i=1}^{M} s^i\sum_{j=0}^{M-i} u_{j,j+i} t^j s^j +  \sum_{i=0}^M u_{i,i} t^i s^i = \\
  &=& \sum_{i=0}^M t^i \tsu_{-i}(ts) + \sum_{i=1}^{M} s^i \tsu_{i}(ts)
  = \sum_{i=-M}^{0} t^{-i} \tsu_{i}(ts) + \sum_{i=1}^{M} s^i \tsu_{i}(ts) =\\
  &=&\sum_{i=-M}^{M} t^{(|i|-i)/2} s^{(|i|+i)/2} \tsu_{i} (ts).
\end{eqnarray*}
\end{proof}

\subsection{Convolution of matrices expressed through convolution of diagonals; banded matrices}

\begin{lemma}\label{lem:ABCgf}	
Let $\bfA=(a_{i,j})_{i,j=0}^A$, $\bfB=(b_{i, j})_{i,j=0}^B$  and $\bfC=\bfA\ast\bfB$.
For a given integer $i$, let $\tsa_i(t)= G_{\diag_i(\bfA)}(t)= \sum_{j=0}^{A-|i|} a_{j,j+i} t^j$ ($|i|\leq A$),
$\tsb_i(t)= G_{\diag_i(\bfB)}(t)$ ($|i|\leq B$) and $\tsc_i(t)=G_{\diag_i(\bfC)}(t)$ ($|i|\leq C = A+B$) be the generating functions
of the diagonals $\diag_i(\bfA)$, $\diag_i(\bfB)$ and $\diag_i(\bfC)$, respectively. Then
\begin{equation}\label{eq:W_L_R3}
  \tsc_i(t) = \sum_{j+k = i} t^{(|j|+|k| - |j+k|)/2} \tsa_j(t) \tsb_k(t)  \;\;(i = -C, \ldots, C)\, .
\end{equation}
\end{lemma}

\begin{proof}
In view of Lemma~\ref{lem:gf2}, we should prove
\begin{eqnarray*}
  \sum_{i=-C}^{C} t^{(|i|-i)/2} s^{(|i|+i)/2} \tsc_{k}(ts) =
  \sum_{i=-C}^{C} t^{(|i|-i)/2} s^{(|i|+i)/2} \sum_{j+k = i} (ts)^{(|j|+|k| - |j+k|)/2} \tsa_j(ts) \tsb_k(ts).
\end{eqnarray*}
We have:
\begin{eqnarray*}
  &&\sum_{i=-C}^{C} t^{(|i|-i)/2} s^{(|i|+i)/2} \tsc_{k}(ts)  \\
  &=&\left(\sum_{j=-A}^{A} t^{(|j|-j)/2} s^{(|j|+j)/2} \tsa_{j}(ts)\right)
  \left(\sum_{k=-B}^{B} t^{(|k|-k)/2} s^{(|k|+k)/2} \tsb_{k}(ts)\right)\\
  &=&
  \sum_{j=-A}^{A} \sum_{k=-B}^{B} t^{(|j|-j)/2} s^{(|j|+j)/2} \tsa_{j}(ts)t^{(|k|-k)/2} s^{(|k|+k)/2} \tsb_{k}(ts)\\
  &=&\sum_{i=-(A+B)}^{A+B} \sum_{j+k=i} t^{(|j|-j)/2} s^{(|j|+j)/2} t^{(|k|-k)/2} s^{(|k|+k)/2} \tsa_{j}(ts)\tsb_{k}(ts)\\
  &=&\sum_{i=-C}^{C} \sum_{j+k=i} t^{(|j|+|k|-(j+k))/2} s^{(|j|+|k|+(j+k))/2} \tsa_{j}(ts)\tsb_{k}(ts).
\end{eqnarray*}
Since $j+k=i$ within the sum, the proof is complete.
\end{proof}

\begin{remark}\label{rem:diag_simplified} For any real  $j$ and $k$, we have:
\begin{eqnarray*}
  {(|j|+|k| - |j+k|)/2} =
  \begin{cases}
    0, & \mbox{if \ \ } jk \ge 0\\
    {|k|} & \mbox{if \ \ }  jk < 0, |j| \ge |k|; \\
    {|j|} & \mbox{if \ \ }  jk < 0, |j| < |k|.
  \end{cases}
\end{eqnarray*}
\end{remark}

The following corollary is a reformulation of  Lemma~\ref{lem:ABCgf} using the explicit form for
the diagonals of $\bfC$.

\begin{corollary}\label{lem:conv}
Let matrices $\bfA$, $\bfB$
and $\bfC=\bfA\ast\bfB$ be as in Lemma~\ref{lem:ABCgf}. Then
the $i$-th diagonal of $\bfC$ is
\begin{equation}\label{eq:conv_diag}
  {\rm diag}_{i}(\bfC) = \sum_{j+k = i}
  \begin{pmatrix}
    0_{(|j|+|k| - |j+k|)/2} \\
    \diag_j(\bfA) \ast \diag_k(\bfB) \\
	0_{(|j|+|k| - |j+k|)/2}
  \end{pmatrix},
\end{equation}
where $i = -C, \ldots, C$, $C=A+B$, and $0_m \in \spR^{m}$ is a vector of zeros of size $m$.
\end{corollary}

In the case of banded matrices, Corollary~\ref{lem:conv} takes the following form.
\begin{corollary}\label{prop:p1p2}	
Let $\bfA=(a_{i,j})_{i,j=0}^A$ be a $(2p_1+1)$-diagonal matrix with $p_1\leq A$, $\bfB=(b_{i, j})_{i,j=0}^B$ be $(2p_2+1)$-diagonal with $p_2\leq B$
and $\bfC=\bfA\ast\bfB$. Then
$\bfC$ is $(2p+1)$-diagonal with $p=p_1 + p_2$ and
the $i$-th diagonal $\diag_i(\bfC)$ of $\bfC$ $(i = -p, \ldots, p)$ is given by \eqref{eq:conv_diag}; its gf is given by \eqref{eq:W_L_R3}.
\end{corollary}

When $p_2 = 0$ so that $\bfB$ is diagonal, Corollary~\ref{prop:p1p2} gives  the following particular case.

\begin{corollary} \label{prop:p_diag}
Let $\bfL$ be a $(2p+1)$-diagonal matrix and $\bfR$ be diagonal with $\bfb$ on the main diagonal. Then $\bfC=\bfA\ast\bfB$ is $(2p+1)$-diagonal with $\diag_i(\bfC) = \diag_i(\bfL) \ast \bfb$, $i = -p, \ldots, p$; in terms of gf, we have $\tsc_i(t) = \tsa_i(t) \tsb_0(t)$, $i = -p, \ldots, p$.
\end{corollary}

\section{Autocovariance matrices and their inverses for AR(1) and AR(2) models}

\subsection{Inverse autocovariance matrices}

It is well known that the inverse to the autocovariance matrix of AR($p$) is positive-definite
symmetric $(2p+1)$-diagonal matrix \cite[p.534]{Shaman1975}.
Below we consider  explicit forms of  $\bm\Sigma^{-1}$ for the AR(1) and AR(2) models.

\subsubsection{AR(1)}
Let $\bm\xi=(\xi_0,\xi_1, \ldots, \xi_{W})^\top$
follows the AR(1) process
\begin{eqnarray}\label{eq:AR1}
  \xi_j=\phi_1\xi_{j-1}+\varepsilon_j,
\end{eqnarray}
where   $j=1, \ldots, W$, {$\phi_1\neq 0$}, and $ \varepsilon_1, \varepsilon_2, \ldots$ are i.i.d. normal random variables ${\cal N} (0,1)$;
$\xi_j$ are ${\cal N} (0,\sigma^2)$ for all $j=0, \ldots, W$, where $\sigma^2=1/(1-\phi_1^2).$
The condition of stationarity of the process \eqref{eq:AR1} is $|\phi_1|<1$.
The covariance matrix $\bm\Sigma$ of the vector $\bm\xi$ is
$\bm\Sigma = \sigma^2( \phi_1^{|i-j|})_{i,j=0}^{W}$.
The inverse of $\bm\Sigma$ is the  tridiagonal matrix {$\bm\Sigma^{-1}= \bfW \in \spR^{(W+1)\times (W+1)}$,
} where
\begin{eqnarray}\label{eq:AR1W}
  \bfW=
  \begin{pmatrix}
    1&k_{1}&0&0  &\ldots\\
    k_1 & k_0    &k_1&0&\ldots\\
    0& k_1    &k_0&k_1&0\\
    \vdots&\ddots&\ddots&\ddots&\ddots&\ddots\\
    && 0    &k_1&k_0&k_1\\
    &&0& 0&k_{1}&1\\
  \end{pmatrix}~,
\end{eqnarray}
$k_0=1+\phi_1^2$, $k_1=-\phi_1$.
Note that the matrix \eqref{eq:AR1W} is defined for any $\phi_1$, not necessarily for $|\phi_1|<1$.

\subsubsection{AR(2)}
Let $\bm\xi=(\xi_0,\xi_1, \ldots, \xi_{W})^\top$
follows the AR(2) process
\begin{eqnarray}\label{eq:AR2}
  \xi_j=\phi_1\xi_{j-1}+\phi_2\xi_{j-2}+\varepsilon_j\; \;(j=2, \ldots, W)
\end{eqnarray}
where   $\varepsilon_j$ are  i.i.d. ${\cal N} (0,1)$
and $\xi_j$ are ${\cal N} (0,\sigma^2)$ for all $j=0, \ldots, W$; here $\sigma^2=(1-\phi_2)/[(1+\phi_1-\phi_2)(1-\phi_1-\phi_2)(1+\phi_2)]$.
 The conditions of stationarity are $\phi_2+|\phi_1|<1$ and $|\phi_2|<1$. The region of stationarity of the AR(2) model is depicted in Figure~1. 
 The inverse of the covariance matrix $\bm\Sigma$ is a  five-diagonal matrix $\bm\Sigma^{-1}= \bfW  =(w_{i,j})_{i,j=0}^W$, where
\begin{eqnarray}\label{eq:AR2W}
  \bfW=
  \begin{pmatrix}
    1&k_{12}&k_2&0  &0&0&\ldots\\
    k_{21}&k_{22}&k_1&k_2&0&0&\ldots\\
    k_2 & k_1    &k_0&k_1&k_2&0&\ldots\\
    0&k_2 & k_1    &k_0&k_1&k_2\\
    \vdots&\ddots&\ddots&\ddots&\ddots&\ddots&\ddots\\
    &&0&k_2& k_1    &k_0&k_1&k_2\\
    &&0&0&k_2&k_1&k_{22}&k_{12}\\
    &&0&0&0& k_2&k_{21}&1\\
  \end{pmatrix}~,
\end{eqnarray}
$k_0=1+\phi_1^2+\phi_2^2$, $k_1=-\phi_1+\phi_1\phi_2$, $k_2=-\phi_2$,
$k_{12}=k_{21}=-\phi_1$, $k_{22}=1+\phi_1^2$.
The matrix \eqref{eq:AR2W} is defined for any $\phi_1$ and $\phi_2$.

\begin{center}
\begin{figure}[h]
\label{fig:ar2}

\begin{tikzpicture}
\centering
\filldraw[black] (-8,-1) circle (0pt) node[anchor=east] {};
\filldraw[black] (-3,-1) circle (1pt) node[anchor=east] {$-1$};
\filldraw[black] (-3,1) circle (1pt) node[anchor=east] {$1$};
\filldraw[black] (-2,-2) circle (1pt) node[anchor=north] {$-2$};
\filldraw[black] (2,-2) circle (1pt) node[anchor=north] {$2$};
\filldraw[black] (0,-2) circle (1pt) node[anchor=north] {$0$};
\draw[->] (-3,-2) -> (-3,2.5) node at (-3,2.5) [right] {$\phi_2$};
\draw[->] (-3,-2) -> (3,-2) node at (3,-2) [right] {$\phi_1$};
\draw[blue, ultra thick,fill=blue!3] (0,1) -- (2,-1) -- (-2,-1) -- cycle;
\filldraw[black] (0,1) circle (3pt) node[anchor=south] {$(0,1)$};
\filldraw[black] (2,-1) circle (3pt) node[anchor=north] {$(2,-1)$};
\filldraw[black] (-2,-1) circle (3pt) node[anchor=north] {$(-2,-1)$};
\end{tikzpicture}

\caption{Region of stationarity of the AR(2) model and critical points}
\end{figure}
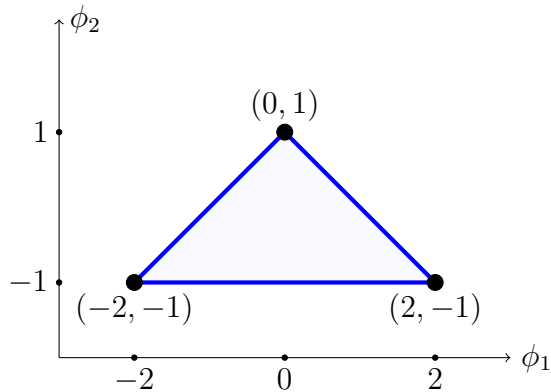
\end{center}

\subsection{Inverses of autocovariance matrices and their generating functions }

Let $\bfW =c \bm\Sigma^{-1}=(w_{i,j})_{i,j=0}^W$, where $\bm\Sigma$ is the autocovariance matrix of AR($p$) with some $p>0$, and $c$ is chosen so that $w_{0,0}=1$ like in \eqref{eq:AR1W} and \eqref{eq:AR2W}.
{Recall that the matrix $\bfW$ is banded.}

Consider the relation between diagonals of $\bfW$ and their generating functions.
Denote $\tsw_k(t) \eqdef G_{\diag_k(\bfW)}(t)$, $t\in\spC$.
For $k\ge 0$, an explicit formula for $\tsw_k$ is  $\tsw_k(t)= \sum_{j=0}^{W-k} w_{j,j+k} t^j$.

\subsubsection{Generating function of a vector of ones}
Below we shall frequently use the following gf.
Let $\bfc_M=(1, \ldots, 1)^\top \in \spR^{M+1}$. Then $$\tsC_M(t) \eqdef G_{\bfc_M}(t) = 1+\ldots+t^M = \frac{1-t^{M+1}}{1-t\;\;} \,.$$
The following convenient formula connects polynomials $\tsC_{M+k}(t)$ and $\tsC_M(t)$ with $k>0$:
\begin{eqnarray}\label{eq:cm}
   \tsC_{M+k}(t) &=& t^k \tsC_{M}(t) + \tsC_{k-1}(t).
\end{eqnarray}

\subsubsection{AR(1)}
Consider the AR(1) model \eqref{eq:AR1} with regression coefficient $\phi_1$. From \eqref{eq:AR1W}, we have
$$
  \tsw_1(t) = -\phi_1 \tsC_{W-1}(t)\;\;{\rm and }\;\;
  \tsw_0(t) = \tsC_{W}(t) + \phi_1^2 t \tsC_{W-2}(t).
$$
Applying \eqref{eq:cm} with $(k,M)=(1,W-1)$ and $(k,M)=(1,W-2)$, we obtain the following lemma.

\begin{lemma}\label{lem:AR1w}
For the AR(1) model \eqref{eq:AR1},
\be\label{eq:AR1w}
  \tsw_1(t) = -\phi_1 \tsC_{W-1}(t), \;\;
  \tsw_0(t) = t \tsC_{W-1}(t) + 1 + \phi_1^2 (\tsC_{W-1}(t)-1).
\ee
\end{lemma}

\subsubsection{AR(2)}
Consider the AR(2) model \eqref{eq:AR2} with regression coefficients $\phi_1$ and $\phi_2$. From \eqref{eq:AR2W}, we obtain
\begin{equation*}
\begin{aligned}
  \tsw_2(t) &= -\phi_2 \tsC_{W-2}(t),\\
  \tsw_1(t) &= -\phi_1 \tsC_{W-1}(t) +\phi_1 \phi_2 t \tsC_{W-3}(t),\\
  \tsw_0(t) &= \tsC_{W}(t) + \phi_1^2 t \tsC_{W-2}(t) + \phi_2^2 t^2 \tsC_{W-4}(t).
\end{aligned}
\end{equation*}

Applying \eqref{eq:cm} with {$k=1$, $M=W-3, W-2$  and with $k=2$, $M=W-4, W-2$}, we obtain the following lemma.

\begin{lemma}\label{lem:AR2w}
For the AR(2) model \eqref{eq:AR2},
\begin{equation}\label{eq:AR2w}
\begin{aligned}
  \tsw_2(t) &= -\phi_2 \tsC_{W-2}(t),\\
  \tsw_1(t) &= -\phi_1 (t \tsC_{W-2}(t)+1) +\phi_1 \phi_2(\tsC_{W-2}(t)-1),\\
  \tsw_0(t) &= t^2 \tsC_{W-2}(t) + t + 1 + \phi_1^2 t \tsC_{W-2}(t) + \phi_2^2 (\tsC_{W-2}(t)-(t+1)).
\end{aligned}
\end{equation}
\end{lemma}

\section{Studying existence of solutions to the problem of blind deconvolution for  the matrices proportional to inverses of covariance matrices in autoregressive models}

For given square matrices $\bfA$ and $\bfB$, we denote  $\tsa_j (t)= G_{\diag_j(\bfA)}(t)$ and $\tsb_j (t) = G_{\diag_j(\bfB)}(t)$.

\subsection{AR(1)}
\label{sec:ar1a}

\begin{theorem}\label{th:AR1_3_1}
Let $W\geq 2$ and  the matrix $\bfW$ be defined by \eqref{eq:AR1W}.
There exist symmetric matrices
$\bfA=(a_{i,j})_{i,j=0}^A$ and  $\bfB=(b_{i, j})_{i,j=0}^B$ with  $A,B \ge 1$
 such that $\bfW= \bfA \ast \bfB$, if and only if  $|\phi_1|= 1$.
\end{theorem}

\begin{proof}
Assume that $\bfW= \bfA \ast \bfB$.
Since $\bfW$ is 3-diagonal, in view of  Corollary~\ref{prop:p_diag} and the assumption of symmetry of $\bfA$  and $\bfB$, we can only consider the case when $\bfA$ is 3-diagonal and $\bfB$ is diagonal.

From \eqref{eq:W_L_R3} we obtain
\begin{eqnarray}\label{eq:AR1wab}
  \tsw_1(t) =  \tsa_1 (t) \tsb_0 (t) ,\;\;
  \tsw_0(t) = \tsa_0(t) \tsb_0(t), \;\; \forall t \in \spC\,.
\end{eqnarray}
Since $\phi_1\neq 0$, from the left equalities in \eqref{eq:AR1w} and \eqref{eq:AR1wab} we deduce that
all the roots of $ \tsC_{W-1}(t)$ are the roots of $\tsa_1(t) \tsb_0 (t)$.
Since $B\ge 1$, at least one of the roots of $ \tsC_{W-1}(t)$  is a root of $\tsb_0 (t)$. Let~$t_1 \in \spC$ be such root.

Suppose that $|\phi_1|\neq 1$.
From the second equality  in  \eqref{eq:AR1wab} we have $\tsw_0(t_1)=0$
but the second equality in \eqref{eq:AR1w} yields $\tsw_0(t_1)=1-\phi_1^2\neq 0$. This contradiction
proves the necessity of $|\phi_1|= 1$.

Assume now  $|\phi_1|= 1$ so that $\phi_1=\pm 1$. In this case,  $\tsw_0(t) = (t+1) \tsC_{W-1}(t)$  and therefore, taking also into account the first equality in \eqref{eq:AR1w}, the equalities \eqref{eq:AR1wab} become
\begin{eqnarray} \label{eq:polyn}
  \tsa_1 (t) \tsb_0 (t) = -\phi_1 \tsC_{W-1}(t), \
  \tsa_0(t) \tsb_0(t) = (t+1) \tsC_{W-1}(t).
\end{eqnarray}
Represent $\tsC_{W-1}(t)$, which is a polynomial of degree  $W-1 \geq  1$,  as a product of two non-zero polynomials $\tsp(t)$  and $\tsq(t)$:
\be \label{eq:productAR1}
\tsC_{W-1}(t)=\tsp(t)\tsq(t)\, .
\ee
Here the polynomial
$\tsp(t)$ can be a constant and  $\tsq(t)$ has degree at least  1. Then we can choose
\be \label{eq:productAR1a}
\tsb_0 (t)= \tsq(t), \;\;\tsa_0(t)=(t+1)\tsp(t)\;\;{\rm  and} \;\;\tsa_1 (t)= -\phi_1 \tsp(t)\, .
\ee
\end{proof}

Assume $|\phi_1|=1$. Let us count the number of different deconvolutions $\bfW= \bfA \ast\bfB$ assuming $a_{0,0}=b_{0,0}=1$ (for any $\bfA $ and $\bfB$ and any $\lambda \neq 0$, we have  $\bfA \ast\bfB = \bfA^\prime \ast\bfB^\prime $ with $\bfA^\prime =\lambda \bfA$ and  $\bfB^\prime =\bfB/\lambda $). This is equivalent to counting the number of different splits \eqref{eq:productAR1} of $\tsC_{W-1}(t)=1+t+\ldots+t^{W-1}$
into a product of polynomials $\tsp(t)$ and  $\tsq(t)$ {with real coefficients}, where $\tsp(0)=\tsq(0)=1$, the degree of $\tsp(t)$ is arbitrary and  $\tsq(t)$ has degree at least  1.  Let us show that this number of splits is $2^M-1$, where $M=\lfloor W/2 \rfloor$.

Assume first that $W-1$ is even; that is, $W-1=2M$. In this case, all roots of $\tsC_{W-1}(t)$ are complex roots of unity and therefore  $\tsC_{W-1}(t)$ is a product of $M$ different quadratic forms which have no real roots. Hence, the total number of required splits  \eqref{eq:productAR1} is $2^M-1$.

Assume now that $W-1$ is odd; that is, $W=2M$. Then $\tsC_{W-1}(t)$ has one real root (which is $-1$) and $2M-2$ complex ones. Therefore  $\tsC_{W-1}(t)$ is a product of $t+1$ and $M-1$ quadratic forms with no real roots. Once again,  the total number of required splits  \eqref{eq:productAR1} is $2^M-1$.

Note that  if we additionally require, as we do in Section~\ref{sec:ABnonnegative}, that the polynomials $\tsp(t)$ and  $\tsq(t)$ in \eqref{eq:productAR1} have non-negative coefficients,
then the total number of eligible splits becomes $H(W)$, the number of different ordered factorizations of $W$ into primes, see \cite{Zhigljavsky.etal2016}.

\subsection{AR(2)}
Consider the AR(2) model \eqref{eq:AR2}  with regression coefficients $\phi_1$ and $\phi_2$.
We aim at identifying matrices $\bfA$ and $\bfB$ so that \eqref{eq:b_decomp} holds; that is, $\bfW= \bfA \ast\bfB$. Since $\bfW$ is 5-diagonal and in view of Corollary~\ref{prop:p1p2} this may only happen in the following  two cases: (a) each of $\bfA$ and $\bfB$ is 3-diagonal
and (b) $\bfA$ is 5-diagonal and $\bfB$ is diagonal.

\subsubsection{$\bfA$ and $\bfB$ are symmetric and both matrices are 3-diagonal}

\begin{theorem}
\label{th:AR2_3_3}
Let $W\ge 3$  and the matrix $\bfW $ be defined in \eqref{eq:AR2W} with $\phi_2 \neq 0$.
Assume that the  matrices
$\bfA=(a_{i,j})_{i,j=0}^A$ and  $\bfB=(b_{i, j})_{i,j=0}^B$ are 3-diagonal and   $A\ge 2$, $B\ge 1$.
There exist such matrices
 $\bfA$ and  $\bfB$ with  $\bfW= \bfA \ast \bfB$, if and only if either (i) $\phi_2 = -1$ and $|\phi_1|\geq 2$ or  (ii) $\phi_2 = 1$ and   $\phi_1=0$.
  \end{theorem}

The proof of Theorem~\ref{th:AR2_3_3} is based on several lemmas.

\begin{lemma}
\label{lem:th33_1}
Under the conditions of Theorem~\ref{th:AR2_3_3}, the existence of matrices $\bfA$ and $ \bfB$ such that $\bfW= \bfA \ast \bfB$, implies that
the polynomial $\tsw_1^2(t) - 4\tsw_2(t) (\tsw_0(t) - 2t \tsw_2(t))$ is the square of a polynomial in $t$. \end{lemma}
\begin{proof}

From \eqref{eq:W_L_R3} we obtain
\be
\label{eq:2AR}
  \tsw_2(t) =\tsa_1(t) \tsb_1(t), \
  \tsw_1(t) =\tsa_0(t) \tsb_1(t) + \tsa_1(t) \tsb_0(t), \
  \tsw_0(t) =\tsa_0(t) \tsb_0(t) + 2t \tsa_1(t) \tsb_1(t).
\ee
For brevity, we will omit the  argument $t$.
We have:
\begin{eqnarray*}
  \tsa_1 \tsb_1 = \tsw_2, \ \tsa_0 \tsb_1 + \tsa_1 \tsb_0 = \tsw_1, \  \tsa_0 \tsb_0  = \tsw_0 - 2t \tsw_2.
\end{eqnarray*}
Denote $\tsy = \tsa_0 \tsb_1$, $\tsz = \tsa_1 \tsb_0$. Then we obtain
\begin{eqnarray*}
  \tsy+\tsz = \tsw_1, \ \tsy\tsz  =  \tsw_2(\tsw_0 - 2t \tsw_2).
\end{eqnarray*}
This means that $\tsy$ and $\tsz$ are the roots of the quadratic equation
$\tsx^2 + \tsw_1 \tsx + \tsw_2(\tsw_0 - 2t \tsw_2)  = 0$. These roots are
\begin{eqnarray}\label{eq:x}
  \tsx_{1,2} = \frac{\tsw_1\pm \sqrt{\tsw_1^2 -4\tsw_2 (\tsw_0 - 2t \tsw_2) }}{2}.
\end{eqnarray}
Since the roots should be polynomials, $\tsw_1^2 - 4\tsw_2(\tsw_0 - 2t \tsw_2) $ should be a square
of a polynomial in $t$ (recall that $\tsw_i$ are polynomials in $t$).
\end{proof}

\begin{lemma}\label{lem:Pol_in_C}
Assume that a polynomial $\tsP_{2M}(t)$ has the form $$\tsP_{2M}(t) = (\tsP_1(t) \tsP_{M-1}(t))^2 + \tsQ_1(t) \tsP_{M-1}(t) + \tsP_{0}$$ for some {$M \ge 1$} (the lower indices mean polynomial degrees).
If the square root of $\tsP_{2M}(t)$ exists then it has the form $\pm\tsP_1(t) \tsP_{M-1}(t) + {\rm const}$.
\end{lemma}

\begin{proof}
Let $\tsP_1(t) = at+b$, $\tsQ_1(t) = dt + f$.
Then for  $\gamma  = d/(2a)$ we have
\begin{equation*}
  \tsP_{2M}(t) = (\tsP_1(t) \tsP_{M-1}(t) + \gamma)^2 + g \tsP_{M-1}(t) + p_0,
\end{equation*}
where $g$ and $p_0$ are some constants.
Let $\tsP_{2M}(t) = \tsU_M^2(t)$. Denote $\tsV_M(t) = \tsP_1(t) \tsP_{M-1}(t) + \gamma$, $\tsD_{M-1}(t) = g \tsP_{M-1}(t) + p_0$.
Then $$\tsD_{M-1}(t) = \tsU_M^2(t) - \tsV_M^2(t) = (\tsU_M(t) - \tsV_M(t))(\tsU_M(t) + \tsV_M(t)).$$
The left-hand side part in this equation is a polynomial of order $M-1$ or less but the right-hand side  part is either identical zero or a polynomial of
order  $M$ or larger. Therefore  this equality can be valid only if $\tsU_M(t) =\pm \tsV_M(t)$.
\end{proof}

\begin{lemma}
\label{lem:th33_2} For the matrix $\bfW $ defined in \eqref{eq:AR2W} with $\phi_2 \neq 0$, the polynomial $\tsw_1^2(t) - 4\tsw_2(t) (\tsw_0(t) - 2t \tsw_2(t))$ is a  polynomial square if and only if at least one of the following relations hold:\\
(A) $\phi_2=-1$, (B) $\phi_2 = 1-\phi_1$,  (C) $\phi_2=1+\phi_1$.
\end{lemma}
\begin{proof}

For brevity, we will omit the polynomial argument $t$. Denote $\tsC = \tsC_{W-2}(t)$.

Direct substitution using \eqref{eq:AR2w} gives
\begin{eqnarray*}
  \tsw_1^2-4(\tsw_0 - 2t \tsw_2) \tsw_2 =
  (-\phi_1 t+\phi_1\phi_2)\tsC -\phi_1(1+\phi_2))^2 + \\
  +4\phi_2(t^2 + (\phi_1^2 + 2\phi_2)t + \phi_2^2) \tsC^2 +
  4\tsC\phi_2(t+1-\phi_2^2(t+1)).
\end{eqnarray*}
Consider this polynomial as a polynomial in $\tsC$. By $\tsv_i$ we denote the coefficient
for $\tsC^i$, $i=0,1,2$.
We have:
\begin{eqnarray*}
  \tsv_0&=&\phi_1^2(1+\phi_2)^2,\\
  \tsv_1&=&2(1+\phi_2)(-\phi_1^2(\phi_2-t) + 2\phi_2(t+1)(1-\phi_2)),\\
  \tsv_2&=&(\phi_1^2+4\phi_2)(t+\phi_2)^2.
\end{eqnarray*}
In view of Lemma~\ref{lem:Pol_in_C}, the determinant
\begin{eqnarray*}
  \tsv_1 - 4 \tsv_0 \tsv_2 = 4(1+\phi_2)^2\phi_2(1+\phi_1-\phi_2)(1-\phi_1-\phi_2)((t+1)^2 \phi_2 + t\phi_1^2)
\end{eqnarray*}
should be equal to zero.
The solutions of the equation $\tsv_1 - 4 \tsv_0 \tsv_2 =0$, with respect to $\phi_2$, are: $\phi_2=0,-1, 1-\phi_1, 1+\phi_1,  -t\phi_1^2/(1+t)^2$.
The root $\phi_2=0$ is inappropriate, since $\phi_2\neq0$ by the definition of AR(2). The root
$\phi_2=-t\phi_1^2/(1+t)^2$ is inappropriate as it depends on $t$.
The proof is complete.
\end{proof}

\begin{proof}[Proof of Theorem~\ref{th:AR2_3_3}]

In view of Lemmas~\ref{lem:th33_1} and \ref{lem:th33_2}, matrices $\bfA$ and $\bfB$ such that
$\bfW= \bfA \ast \bfB$ holds may exist only in the three particular cases indicated in Lemma~\ref{lem:th33_2}; that is,
(A) $\phi_2=-1$, (B) $\phi_2 = 1-\phi_1$,  (C) $\phi_2=1+\phi_1$. Let us consider these three cases separately.

(A) Assume $\phi_2 = -1$. Then
\begin{eqnarray*}
  \tsw_2(t) =  \tsC_{W-2}(t),\
  \tsw_1(t) = -\phi_1 \tsC_{W-2}(t)(t+1), \
  \tsw_0(t) = \tsC_{W-2}(t)(t^2 + \phi_1^2 t + 1).
\end{eqnarray*}
From \eqref{eq:AR2w} and \eqref{eq:2AR} we then have:
\begin{equation}\label{eq:three_eq}
  \begin{aligned}
  \tsa_1(t) \tsb_1(t) &= \tsC_{W-2}(t), \\
  \tsa_0(t) \tsb_1(t) + \tsa_1(t) \tsb_0(t) &= -\phi_1 \tsC_{W-2}(t)(t+1), \\
  \tsa_0(t) \tsb_0(t) + 2t \tsa_1(t) \tsb_1(t) &=  \tsC_{W-2}(t)(t^2 + \phi_1^2 t + 1).
  \end{aligned}
\end{equation}
Since polynomials $\tsa_0(t)$ and $\tsb_0(t)$ have one degree higher than the polynomials $\tsa_1(t)$ and $\tsb_1(t)$ respectively,
 $\tsa_1(t) \tsb_1(t) = \tsC_{W-2}(t)$ and $\tsa_0(t) \tsb_0(t)$ contains $\tsC_{W-2}(t)$ as a multiplier,  we obtain that $\tsa_0(t)=(\lambda_0+\lambda_1t)\tsa_1(t)$ and $\tsb_0(t)=(\mu_0+\mu_1t)\tsb_1(t)$  for some $\lambda_0,\lambda_1,\mu_0$ and $\mu_1$.
Substituting this into equations \eqref{eq:three_eq} and cancelling $C_{W-2}(t)$, which is a common multiplier in all equations, we obtain the following two equations for $\lambda_0,\lambda_1,\mu_0$ and $\mu_1$:
\begin{eqnarray*}
  \left\{
  \begin{array}{lcl}
    (\mu_0+\mu_1t)\,(\lambda_0+\lambda_1t) &=& t^2-2t+\phi_1^2t +1 \\
    (\mu_0+\mu_1t)+(\lambda_0+\lambda_1t) &=& -(t+1)\phi_1
  \end{array}
  \right.
\end{eqnarray*}
Equating the coefficients of the two polynomials in $t$ we find that
there are no solutions  for $\lambda_0,\lambda_1,\mu_0$ and $\mu_1$ when $|\phi_1|<2$. On the other hand, there are the following solutions when $|\phi_1|\geq 2$: let
\be
\label{eq:z}
z_{1}= \frac{ -\phi_1 + \sqrt{\phi_1^2-4}}{2}\, \;\; z_{2}= \frac{ -\phi_1 - \sqrt{\phi_1^2-4}}{2}
\ee
be two solutions of the equation $z^2+\phi_1 z+1=0$, then we can choose either $
\mu_0 = \lambda_1= z_1,\;\;\lambda_0 = \mu_1 = z_2
$
or
$
\mu_0 = \lambda_1= z_2,\;\;\lambda_0 = \mu_1 = z_1.
$
This gives the required expressions for $\tsa_0(t)$, $\tsa_1(t)$, $\tsb_0(t)$ and $\tsb_1(t)$, see Corollary~\ref{cor:AR1_3_1ab3}.

(B) Assume $\phi_2 = 1 - \phi_1$. Then the equations \eqref{eq:AR2w} become
\begin{equation}
\label{eq:aaa}
  \begin{aligned}
  \tsa_1(t) \tsb_1(t) &= -(1-\phi_1) \tsC_{W-2}(t),\\
  \tsa_0(t) \tsb_1(t) + \tsa_1(t) \tsb_0(t)  &= -\phi_1 (t \tsC_{W-2}(t)+1) +\phi_1 (1-\phi_1)(\tsC_{W-2}(t)-1),\\
 \!\!\!\!\!\! \tsa_0(t) \tsb_0(t)\! + \!2t \tsa_1(t) \tsb_1(t) &= t^2 \tsC_{W-2}(t) \!+\! t\! +\! 1\! +\! \phi_1^2 t \tsC_{W-2}(t) \!+\! (1\!-\!\phi_1)^2 (\tsC_{W-2}(t)\!-\!(t\!+\!1)).
\end{aligned}
\end{equation}
{Using the expression \eqref{eq:x} for the products $\tsa_0(t) \tsb_1(t)$ and $\tsa_1(t) \tsb_0(t)$, we obtain $\tsa_1 (t)\tsb_0(t)= (1-\phi_1)(t+1)\tsC_{W-2}(t)$ as one of the two roots given by \eqref{eq:x}.
Similar to the case (A),  
 from the first equation in \eqref{eq:aaa} we obtain  $\tsb_0(t)=-(1+t)\tsb_1(t)$.
Take any root $t_1$ of $\tsa_1 (t)$ and substitute it into the second equation in \eqref{eq:aaa}.
Since the roots of $\tsa_1 (t)$ consist of the roots of $\tsC_{W-2}(t)$, we obtain $\phi_1(2-\phi_1)=0$.}
The solution $\phi_1=0$ implies $\phi_2 = 1$ and $\phi_1=2$ gives  $\phi_2 = -1$. The case $\phi_1=2,\phi_2 = -1$ gives a solution  described above in the case (A). In the case $\phi_1=0,\phi_2 = 1$ we have  solutions to \eqref{eq:2AR} obtained from  arbitrary splitting
\be
\label{eq:AR2_eq}
\tsC_{W-2}(t)=\tsp(t)\tsq(t)
\ee
of $\tsC_{W-2}(t)$
into a product of two non-zero polynomials $\tsp(t)$ and $\tsq(t)$  and setting \be
\label{eq:split_again}
\mbox{$\tsa_0(t)=(1+t)\tsp(t)$,
$\tsa_1(t)=\tsp(t)$, $\tsb_0(t)=(1+t)\tsq(t)$ and $ \tsb_1(t)=-\tsq(t)$.}
\ee

(C) Assume $\phi_2 = 1 + \phi_1$.  The equations \eqref{eq:AR2w} become
\begin{equation*}
  \begin{aligned}
  \tsa_1(t) \tsb_1(t) &= -(1+\phi_1) \tsC_{W-2}(t),\\
  \tsa_0(t) \tsb_1(t) + \tsa_1(t) \tsb_0(t)  &= -\phi_1 (t \tsC_{W-2}(t)+1) +\phi_1 (1+\phi_1)(\tsC_{W-2}(t)-1),\\
  \tsa_0(t) \tsb_0(t) + 2t \tsa_1(t) \tsb_1(t) &= t^2 \tsC_{W-2}(t) + t + 1 + \phi_1^2 t \tsC_{W-2}(t) + (1+\phi_1)^2 (\tsC_{W-2}(t)-(t+1)).
  \end{aligned}
\end{equation*}
One of the two solutions for $\tsa_1 (t)\tsb_0(t) $ is $\tsa_1 (t)\tsb_0(t)= \tsC_{W-2}(t+(\phi_1+1)^2)$.
Similarly to the above we obtain $\phi_1(2+\phi_1)=0$. The solution $\phi_1=0$ gives $\phi_2 = 1$ and hence the same set of solutions to equations \eqref{eq:AR2w} as in Case (B).
The solution $\phi_1=-2$ gives $\phi_2 = -1$ and is covered in Case (A).
\end{proof}

\begin{corollary}\label{cor:AR1_3_1ab3}
Summarizing the findings in case (i) of Theorem~\ref{th:AR2_3_3},  when $\phi_2=-1$, the  solutions to \eqref{eq:2AR} exist when  $|\phi_1|\geq 2$ and can be constructed as follows:
make an  arbitrary split
\eqref{eq:AR2_eq}
of $\tsC_{W-2}(t)$
into a product of two non-zero polynomials $\tsp(t)$ and $\tsq(t)$,  set
$\tsa_1(t)=\tsp(t)$, $ \tsb_1(t)=\tsq(t)$ and
\be
{\rm either}\;\;
\left\{
  \begin{array}{c}
   \tsa_0(t)=(z_2+z_1t)\tsp(t) \\
 \tsb_0(t)=(z_1+z_2t)\tsq(t) \\
  \end{array}
\right.
\;\;{\rm or}\;\;
\left\{
  \begin{array}{c}
   \tsa_0(t)=(z_1+z_2t)\tsp(t) \\
 \tsb_0(t)=(z_2+z_1t)\tsq(t) ,\\
  \end{array}
\right.
\ee
where $z_1$ and $z_2$ are defined in \eqref{eq:z}.
\end{corollary}

\subsubsection{$\bfA$ is symmetric and 5-diagonal, $\bfB$ is diagonal}

\begin{theorem}\label{th:AR2_5_1}
Let $W\geq 4$ and the matrix $\bfW $ be defined in \eqref{eq:AR2W}. There exist matrices
$\bfA=(a_{i,j})_{i,j=0}^A$ and  $\bfB=(b_{i, j})_{i,j=0}^B$ with  $A\geq 2$ and $B \geq 2$, where $\bfA$ is 5-diagonal and  $\bfB$ is diagonal,
such that $\bfW= \bfA \ast \bfB$ if and only if  either (a) $\phi_2=-1$ or (b)  $\phi_1=0$ and $\phi_2=1$.
\end{theorem}

\begin{proof}
Assume that $\bfW= \bfA \ast \bfB$.
From \eqref{eq:W_L_R3} we obtain
\begin{eqnarray}\label{eq:ar2_diag}
  \tsw_2(t) =  \tsa_2 (t) \tsb_0 (t) ,\;\;
  \tsw_1(t) =  \tsa_1 (t) \tsb_0 (t) ,\;\;
  \tsw_0(t) = \tsa_0(t) \tsb_0(t), \;\; \forall t \in \spC\,.
\end{eqnarray}
Since $\phi_2\neq 0$, from the left equalities in \eqref{eq:AR2w} and \eqref{eq:ar2_diag} we deduce that
all the roots of $ \tsC_{W-2}(t)$ are the roots of $\tsa_2(t) \tsb_0 (t)$.
Since by assumption $B\ge 2$, at least two of the roots of $ \tsC_{W-2}(t)$  are the roots of $\tsb_0 (t)$. Let $t_1$ and $t_2$ be two of these roots.

Suppose that the conditions on $\phi_1$ and $\phi_2$ are not fulfilled.
From the second equality  in~\eqref{eq:AR2w}, $\tsw_1(t_i)= 0$
but the second equality in \eqref{eq:ar2_diag} yields $\tsw_1(t_i)=-\phi_1(1+\phi_2)$, $i=1,2$.
From the third equality  in \eqref{eq:AR2w} $\tsw_0(t_i)= 0$
but the third equality in \eqref{eq:ar2_diag} yields $\tsw_0(t_i)=(1-\phi_2^2)(t_i+1)$, $i=1,2$.
This is possible only in the cases (a) $\phi_2=-1$ and (b) $\phi_1=0$, $\phi_2=1$. This contradiction
proves the necessity of the conditions on $\phi_1$ and $\phi_2$.

Assume (a): $\phi_2 = -1$. Then
\begin{eqnarray*}
  \tsw_2(t) =  \tsC_{W-2}(t),\
  \tsw_1(t) = -\phi_1 (\tsC_{W-2}(t)(t+1), \
  \tsw_0(t) = \tsC_{W-2}(t)(t^2 + \phi_1^2 t + 1).
\end{eqnarray*}
Therefore,
\begin{eqnarray*}
  \tsa_0(t) \tsb_0(t) &=& \tsC_{W-2}(t)(t^2 + \phi_1^2 t + 1), \\
  \tsa_1(t) \tsb_0(t) &=& -\phi_1 (\tsC_{W-2}(t)(t+1), \\
  \tsa_2 (t) \tsb_0 (t) &=& \tsC_{W-2}(t).
\end{eqnarray*}
 Let $\tsC_{W-2}(t)=\tsp(t)\tsq(t)$, where $\tsp(t)$ is a non-zero polynomial of any degree including 0 and $\tsq(t)$ is a polynomial of degree at least 2. Then we can choose
\be
\label{eq:ar4a}
\mbox{
$\tsb_0 (t)= \tsq(t)$, $\tsa_2 (t)= \tsp(t)$, $\tsa_1(t)=-\phi_1(t + 1)\tsp(t)$ and $\tsa_0(t)=(t^2 + \phi_1^2 t + 1)\tsp(t)$.}
\ee
 Now assume (b): $\phi_1=0$, $\phi_2=1$. Then
\begin{eqnarray*}
  \tsw_2(t) = - \tsC_{W-2}(t),\
  \tsw_1(t) = 0,\
  \tsw_0(t) = (t^2+1)\tsC_{W-2}(t).
\end{eqnarray*}
Therefore,
\begin{eqnarray*}
  \tsa_0(t) \tsb_0(t) =(t^2+1)\tsC_{W-2}(t), \;\;
  \tsa_1(t) \tsb_0(t) = 0, \;\;
  \tsa_2 (t) \tsb_0 (t) = - \tsC_{W-2}(t).
\end{eqnarray*}
Again, let $\tsC_{W-2}(t)=\tsp(t)\tsq(t)$ with the same assumptions on $\tsp(t)$ and $\tsq(t)$. Then we can choose
\be
\label{eq:ar4}
\mbox{$\tsb_0 (t)= \tsq(t)$, $\tsa_2 (t)= -\tsp(t)$, $\tsa_1(t)= 0$  and $\tsa_0(t)=(t^2 + 1)\tsp(t)$.}
\ee
\end{proof}

\begin{remark}\label{rem:B2}
For technical reason, Theorem~\ref{th:AR2_5_1} does not cover the case $B=1$. This case can be treated separately as follows.
First, similarly to the discussion in Section~\ref{sec:ar1a}, the polynomial $\tsC_{W-2}(t)$ is divisible by a polynomial $\tsq(t)=\tsb_0 (t)$ of degree 1 {with real coefficients} if and only if $W$ is {odd}; this polynomial is $\tsb_0 (t)=\lambda(1+t)$ with $\lambda \neq 0$.
Consider, for {odd} $W$,  the equalities \eqref{eq:AR2w} and \eqref{eq:ar2_diag}. The first and third equalities in \eqref{eq:AR2w} show that
{$\tsa_0 (t)=\tsw_0(t)/\tsb_0 (t)$ and $\tsa_2 (t)=\tsw_2(t)/\tsb_0 (t)$ are polynomials for arbitrary values of $\phi_1$ and $\phi_2$. However,
the second  equality in \eqref{eq:AR2w} implies that $\tsa_1 (t)=\tsw_1(t)/\tsb_0 (t)$ is a polynomial
if and only if $\tsw_1(-1) = 0$, that is,} either $\phi_2=-1$ or $\phi_1=0$ (in the latter case, $\tsa_1 (t)=0$).
\end{remark}

\begin{remark}\label{rem:remainder}The main steps in the proofs of Theorems~\ref{th:AR1_3_1} and \ref{th:AR2_5_1} is  calculation of  remainders
of division of $\tsw_k(t)$ by $\tsC_{W-p}(t)$,  where $p=1$ and $k=0$ for AR(1) and $p=2$ and $k=0,1$ for AR(2). Indeed, denote such a reminder as $\tsr_k(t)$ and $B$ roots of $\tsb_0 (t)$, which are also the roots of    $\tsC_{W-p}(t)$, as $t_1,\ldots,t_p$ (these roots exist since $B\ge p$). Then $\tsw_k(t_i) = \tsr_k(t_i)$ for any $i=1,\ldots,p$.
If $p$ is larger than the degrees of the remainders $\tsr_k(t)$ for each $k$, then $\{t_i\}_{i=1}^p$ are the roots of $\tsw_k(t)$ if and only if the remainders $\tsr_k(t_i)$ are zero. This is the case of Theorem~\ref{th:AR1_3_1} and Theorem~\ref{th:AR2_5_1}.
\end{remark}

\begin{remark}\label{rem:not_existance}
Theorems~\ref{th:AR1_3_1}, \ref{th:AR2_3_3} and \ref{th:AR2_5_1} show that for stationary  AR(1) and AR(2) models the deconvolution
$\bfW=\bfA \ast \bfB$ with $A,B \geq 2$
cannot be performed  as  the necessary and sufficient conditions in these theorems contradict to the stationarity conditions.
However, in view of Remark~\ref{rem:B2}, if $W$ is {odd} and $B=1$ then the deconvolutions $\bfW=\bfA \ast \bfB$ exist when $\phi_1=0$ and $\phi_2$ is arbitrary.
\end{remark}

\section{Construction of non-negative definite matrices $\bfA$ and~$\bfB$ such that $\bfW= \bfA \ast \bfB$}
\label{sec:ABnonnegative}

In this section, we additionally assume that symmetric matrices $\bfA$ and~$\bfB$ such that $\bfA \ast \bfB=\bfW$  are non-negative definite.
The matrix $\bfW$ is as in the previous section; that is, it is given by either \eqref{eq:AR1W} or \eqref{eq:AR2W}.

\subsection{AR(1)}
\label{sec:ar1_pos}

Assume that $\mathbf{W}$ has the form \eqref{eq:AR1W}.
Under the conditions of Theorem~\ref{th:AR1_3_1},
symmetric matrices $\bfA$ and $\bfB$ with $\bfW= \bfA \ast \bfB$  exist if and only if $|\phi_1|=1$.
A general method of construction of such $\bfA$ and $\bfB$ for $\phi_1=\pm 1$ is given by formulas \eqref{eq:productAR1} and \eqref{eq:productAR1a} at the end of the proof of Theorem~\ref{th:AR1_3_1} in terms of the generating functions of the two diagonals of $\bfA$ and one diagonal of $\bfB$ (recall that the matrix $\bfB$ is diagonal).

In the present case, where we require
 matrices $\bfA$ and $\bfB$ to be non-negative definite,  we need the following lemma.

 \begin{lemma}
 \label{lem:polyn}
 Let $A\geq 1$ be a positive integer and
 $\tsp(t)= c_0 + c_1 t + \ldots+ c_{A-2}t^{A-2}+c_{A-1}t^{A-1}$ be a polynomial of degree $A-1\geq 0$.
 Consider a symmetric 3-diagonal matrix $\bfA=(a_{i,j})_{i,j=0}^A$ defined by~\eqref{eq:productAR1a}; that is, a matrix with the gf of the main diagonal $\tsa_0(t)=(t+1)\tsp(t)$ and the gf of
 the first diagonal $\tsa_1 (t)= G_{\diag_1(\bfA)}(t)= \beta \tsp(t)$, where $\beta=\pm 1$. For this matrix and for any vector $\bfz=(z_0, \ldots, z_A)^\top \in \mathbb{R}^{A+1}$ we have
\be \label{eq:productAR1b}
\bfz^\top \bfA \bfz= \sum_{i=0}^{A-1}c_i(z_i+\beta z_{i+1})^2 \, .
\ee
 \end{lemma}
\begin{proof}
Since $\tsa_0(t)=(t+1)\tsp(t)$, the diagonal elements of $\bfA$ are $a_{i,i}=c_i+c_{i-1}$, where $i=0,1, \ldots, A$ and we set $c_{-1}=c_A=0.$
The elements on the first diagonal of $\bfA$ are $a_{i,i+1}= \beta c_i$ ($i=0,1, \ldots, A-1$). We also have  $a_{i,i+1}=a_{i+1,i}$ and $a_{i,j}=0$ for $|i-j|>1$.
Therefore, for any $\bfz=(z_0, \ldots, z_A)^\top $ we obtain
$$
\bfz^\top \bfA \bfz=\sum_{i,j=0}^A z_ia_{i,j}z_j=c_0z_0^2\!+\!c_{A-1}z_A^2\!+\!\sum_{i=1}^{A-1} z_i^2 (c_{i-1}\!+\!c_i)\!+\! 2\beta\sum_{i=0}^{A-1}c_i z_i z_{i+1}= \sum_{i=0}^{A-1}c_i(z_i\!+\!\beta z_{i+1})^2 \, .
$$
\end{proof}

Theorem~\ref{th:AR1_3_1} and Lemma~\ref{lem:polyn} with $\beta=-\phi_1$ yield the following corollary.

\begin{corollary}\label{cor:AR1_3_1ab}
Let $W\geq 2$ and  the matrix $\bfW$ be defined by \eqref{eq:AR1W}.
There exist non-negative definite symmetric matrices
$\bfA=(a_{i,j})_{i,j=0}^A$ and  $\bfB=(b_{i, j})_{i,j=0}^B$ with  some $A,B \ge 1$
 such that $\bfW= \bfA \ast \bfB$ if and only if  $|\phi_1|= 1$ and both polynomials $\tsp(t)$ and $\tsq(t)$ in
 \eqref{eq:productAR1} are non-constants and have non-negative coefficients.
\end{corollary}

\begin{proof}
It follows from \eqref{eq:productAR1b} that the matrix $\bfA$ defined in Lemma~\ref{lem:polyn} is non-negative definite if and only if
all coefficients $c_i$ of the polynomial $\tsp(t)$ are non-negative. Since  $\bfB$ is diagonal, from the first equation in \eqref{eq:productAR1a} we obtain that
$\bfB$  is non-negative definite if and only if
all coefficients  of the polynomial $\tsq(t)$ are non-negative.
\end{proof}

Since the polynomial $\tsp(t)$ in \eqref{eq:productAR1} can have zero degree (this would correspond to a $2 \times 2$ matrix~$\bfA$), the split \eqref{eq:productAR1} with both polynomials having non-negative coefficients, can always be made. If we assume $A>1$, then both polynomials
$\tsp(t)$ and $\tsq(t)$  in \eqref{eq:productAR1} have to have degree at least one and the problem of  construction of non-constant  polynomials $\tsp(t)$ and $\tsq(t)$ with  non-negative coefficients
 such that  \eqref{eq:productAR1} holds becomes more difficult. This problem was studied in \cite{Zhigljavsky.etal2016}.
  Theorem 1 in \cite{Zhigljavsky.etal2016} states that, under the additional condition $\tsp(0)=1$, if such polynomials  exist then all their  coefficients are either zeros or ones. Corollary 3 in \cite{Zhigljavsky.etal2016} implies that such polynomials exist if and only if $W$ is not prime. Paper \cite{Zhigljavsky.etal2016} also provides  different  schemes of construction of such polynomials $\tsp(t)$ and $\tsq(t)$ for composite  $W$.

\subsection{AR(2): $\bfA$ and $\bfB$ are symmetric and both matrices are 3-diagonal}

In view of Theorem~\ref{th:AR2_3_3} we need to consider the following two cases only:  (i) $\phi_2 = -1$ and $|\phi_1|\geq 2$ and  (ii)
$\phi_1=0,\phi_2 = 1$.

In  case (i), when $\phi_2 = -1$ and $|\phi_1|\geq 2$, the way of constructing   solutions to \eqref{eq:2AR}
is described in Corollary~\ref{cor:AR1_3_1ab3}. Under the additional assumption $|\phi_1|= 2$ we have for $z_1,z_2$ from \eqref{eq:z}:
 $z_1=z_2=-1$ for $\phi_1=2$ and $z_1=z_2=1$ for $\phi_1=-2$. Similarly to Corollary~\ref{cor:AR1_3_1ab}, Lemma~\ref{lem:polyn} and Corollary~\ref{cor:AR1_3_1ab3} yield the following.

\begin{corollary}\label{cor:AR1_3_1ab5}
Let $W\geq 2$ and  the matrix $\bfW$ be defined by \eqref{eq:AR2W} with  $|\phi_1|= 2$ and $\phi_2 = -1$.
Make a split \eqref{eq:AR2_eq} of $\tsC_{W-2}(t)$ into a product of two polynomials $\tsp(t)$ and $\tsq(t)$ with  non-negative coefficients.
Define  3-diagonal  matrices $\bfA$ and $\bfB$ by the following generating functions of their diagonals:
 $\tsa_0(t)=(1+t)\tsp(t)$, $ \tsb_0(t)=(1+t)\tsq(t)$, $\tsa_1(t)= \beta \tsp(t)$ and  $ \tsb_1(t)=\beta \tsq(t)$, where $\beta=-\phi_1/2 = -{\rm sign}(\phi_1)$. Then both matrices $\bfA$ and $\bfB$ are
 non-negative definite and $\bfW= \bfA \ast \bfB$.
 \end{corollary}

If $|\phi_1|>2$ it does not seem possible to find a general scheme of construction of non-negative definite matrices $\bfA$ and $\bfB$ satisfying $\bfW= \bfA \ast \bfB$.

In  case (ii), when $\phi_1=0,\phi_2 = 1$, general   solutions to \eqref{eq:2AR} are obtained from  \eqref{eq:AR2_eq}, an arbitrary splitting of $\tsC_{W-2}(t)$
into a product of two polynomials $\tsp(t)$ and $\tsq(t)$  and using \eqref{eq:split_again}.
Similarly to Corollary~\ref{cor:AR1_3_1ab}, Lemma~\ref{lem:polyn} yields the following: to guarantee that
 $\bfA$ and $\bfB$ are
 non-negative definite we simply make an additional requirement that  the polynomials $\tsp(t)$ and $\tsq(t)$ in
 \eqref{eq:AR2_eq} have non-negative coefficients. Lemma~\ref{lem:polyn} again justifies the non-negative definiteness of  $\bfA$ and $\bfB$.

Similarly to the discussion at the end of Section~\ref{sec:ar1_pos}, we may use the results of    \cite{Zhigljavsky.etal2016}  for establishing the existence
 of non-constant polynomials $\tsp(t)$ and $\tsq(t)$ with non-negative coefficients in
 \eqref{eq:AR2_eq} and the building schemes for the corresponding matrices $ \bfA$ and $ \bfB$ of size larger than $2\times2$. The required non-trivial split \eqref{eq:AR2_eq}
 into non-constant polynomials $\tsp(t)$ and $\tsq(t)$ with  non-negative coefficients exists if and only if $W-1$ is a composite number.

\begin{remark}
The particular cases of the AR(2) model considered in this section (where the deconvolution $\bfW= \bfA \ast \bfB$ can be performed and $\bfA$ and $\bfB$ are   non-negative definite 3-diagonal matrices)  are exactly  the cases where the  pairs of parameters $(\phi_1,\phi_2)$ of the AR(2) model are at critical points of the region of stationarity of the AR(2) model, see Figure 1.
\end{remark}

\subsection{AR(2): $\bfA$ is symmetric and 5-diagonal, $\bfB$ is diagonal}

Similarly to the results above, it can be shown that a 5-diagonal symmetric matrix $\bfA$ with generating functions of the three diagonals
$\tsa_2 (t)= \tsp(t)$, $\tsa_1(t)=\beta_1(t + 1)\tsp(t)$ and $\tsa_0(t)=(t^2 + \beta^2 t + 1)\tsp(t)$
is non-negative definite if  the coefficients of the polynomial $\tsp(t)$ are non-negative and $|\beta|\leq 1$. This implies
that, if the polynomials  $\tsp(t)$  and $\tsq(t)$ have non-negative coefficients then the two  constructions of Theorem~\ref{th:AR2_5_1},  \eqref{eq:ar4a}  and \eqref{eq:ar4} with $|\phi_1|\leq 1$, lead to non-negative definite matrices $\bfA$ and $\bfB$.

\section{Non-existence of the deconvolution  $\bfW\!=\!\bfA\!\ast\!\bfB$  for a stationary AR($p$) model with general $p$ and  diagonal~$\bfB$}

Consider the  AR($p$) model in the form
$\sum_{i=0}^p \alpha_i\xi_{n-i} = \varepsilon_n$,
where $\Var \varepsilon_n=1$ and $\alpha_p\neq 0$.
This notation is connected to the basic notation in the form
$\xi_{n} = \sum_{i=1}^p \phi_i\xi_{n-i} + \varepsilon_n$ by  simple relations $\alpha_0=1$ and $\alpha_i=-\phi_i$ for $i>0$.

Let $\bm\Sigma$ be the autocovariance matrix of this AR($p$) process. In accordance with \cite{Siddiqui1958} and \cite[Eq.~10]{Shaman1975},  the matrix $\bfW = \bm\Sigma^{-1}$  has the form
of a symmetric $(2p+1)$-diagonal matrix with
\begin{eqnarray*}
  w_{i,j}=
  \begin{cases}
  \sum_{m=0}^{\min(i,j)} \alpha_m \alpha_{m+|i-j|}\ \text{for}\ \max(i,j) < p;\\
  \sum_{m=0}^{p - |i-j|} \alpha_m \alpha_{m+|i-j|}\ \text{for}\ \max(i,j)\geq p,\ \min(i,j)\leq W-p,\  |i-j|=0,\ldots, p;\\
  \sum_{m=0}^{\min(W-i,W-j)} \alpha_m \alpha_{m + |i-j|}\ \text{for}\ \min(i,j) > W-p;\\
  0,\ \text{otherwise,}
  \end{cases}
\end{eqnarray*}
where $i,j = 0,\ldots, W$.

One of the stationarity conditions for the AR($p$) model is
$|\phi_p|<1$, see e.g. Jury's test of stability \cite[Section 3.9]{Jury1964}.
In this section, we will show that if $|\phi_p|<1$ then the deconvolution \eqref{eq:b_decomp} is impossible under the assumption that the matrix $\bfB$  is diagonal and has a size larger than $p\times p$.

In the same manner as before, we can express the diagonals of $\bfW$
through the generating functions $\tsC_M(t)$ of vectors $(1,\ldots,1)^\top\in \spR^{M+1}$
with $M=W-2p,W-2p+1,\ldots,W$:
\begin{eqnarray}\label{eq:ARdiag}
  \tsw_k(t) = \sum_{i=0}^{p-k} \alpha_i \alpha_{i+k}t^i \tsC_{W-2i-k},\ k=0,\ldots,p.
\end{eqnarray}
For example, $\tsw_p(t) = -\phi_p \tsC_{W-p}(t)$.

The expressions for $\tsw_k(t)$ in terms of $\tsC_{W-p}(t)$ play an important role for obtaining
the results for AR(1) and AR(2). In particular,
Theorems~\ref{th:AR1_3_1} and \ref{th:AR2_5_1}, where either $\bfA$ or $\bfB$  in the
deconvolution  $\bfW=\bfA\ast\bfB$  is diagonal, are based on the consideration of remainders of division of $\tsw_k(t)$ by $\tsC_{W-p}(t)$ for $k=0,1$ and AR(1)
and $k=0,1,2$ and AR(2), see Remark~\ref{rem:remainder}.

It follows from \eqref{eq:ARdiag} that $\tsw_p(t)$ is proportional to $\tsC_{W-p}(t)$ for any $p$.
Therefore, it can be proved,  similarly to the  proofs of Theorems~\ref{th:AR1_3_1} and \ref{th:AR2_5_1},
that a necessary condition for the possibility of deconvolution
$\bfW=\bfA\ast\bfB$, where $\bfB$ is a diagonal matrix and {$B>p$}, is zero remainders of
division of $\tsw_k(t)$ by $\tsC_{W-p}(t)$ for $k=0,\ldots,p-1$. The proof uses the equalities
\begin{eqnarray*}
  \tsw_k (t) = \tsa_k(t) \tsb_0(t),\ k\leq p,
\end{eqnarray*}
given in Corollary~\ref{prop:p_diag}.

The following lemma calculates the remainders.

\begin{lemma}\label{lem:remainder}
The remainder $\tsr_k(t)$ of the division of $\tsw_k(t)$ by $\tsC_{W-p}(t)$ for $k=0,\ldots, p$ is equal to
\begin{eqnarray}\label{eq:res}
  \tsr_k(t)=\sum_{i=0}^{p-k} \alpha_i \alpha_{i+k}(\tsC_{p-k-i-1}(t) - \tsC_{i-1}(t)).
\end{eqnarray}
(Here we assume that $\tsC_{-1}=0$.)
\end{lemma}

\begin{proof}
Directly follows from \eqref{eq:ARdiag} and the equalities $t^i \tsC_{W-2i-k}= \tsC_{W-i-k} - \tsC_{i-1}$
and $\tsC_{W-i-k} = t^{p-(i+k)} \tsC_{W-p} + \tsC_{p-k-i-1}$, which are particular cases of \eqref{eq:cm}.
\end{proof}

\begin{corollary}\label{cor:tsr}
\begin{eqnarray}
  \label{eq:res_p}
  \tsr_p(t)&=&0,\\
  \label{eq:res_p1}
  \tsr_{p-1}(t) &=& \alpha_0\alpha_{p-1} - \alpha_1\alpha_{p},\\
  \label{eq:res_0}
  \tsr_0(t) &=& \sum_{m=0}^{p-1} \left(\sum_{i=0}^{\min(m, p-m-1)}(\alpha_i^2-\alpha_{p-i}^2)\right) t^m.
\end{eqnarray}
\end{corollary}

\begin{proof}
The equality \eqref{eq:res_p} is readily obtained from \eqref{eq:res}, since $\tsC_{-1}=0$.
The equality \eqref{eq:res_p1} is proved by direct substitution of $k=p-1$ and $i=0, 1$ to
\eqref{eq:res}.

For the proof of \eqref{eq:res_0}, let us substitute $k=0$ to the sum \eqref{eq:res}:
\begin{eqnarray}\label{eq:res0a}
  \tsr_0(t)=\sum_{i=0}^{p} \alpha_i^2(\tsC_{p-i-1}(t) - \tsC_{i-1}(t))
\end{eqnarray}
and then write this sum as a polynomial in $t$.
By \eqref{eq:cm} we obtain for $i=0,\ldots,p$
\begin{eqnarray*}
  \tsC_{p-i-1}(t) - \tsC_{i-1}(t) =
  \begin{cases}
  t^i \tsC_{p-2i-1}, \ i<p/2,\\
  0, \ i=p/2,\ p \mbox{\ is even},\\
  -t^{p-i} \tsC_{2i-p-1}, \ i>p/2.
  \end{cases}
\end{eqnarray*}
Therefore,
\begin{eqnarray*}
  \tsr_0(t)=\sum_{i=0}^{\lfloor (p-1)/2\rfloor} \alpha_i^2 t^i  \tsC_{p-2i-1} - \sum_{j=\lceil (p+1)/2\rceil }^{p} \alpha_j^2 t^{p-j} \tsC_{2j-p-1}.
\end{eqnarray*}
Taking $j=p-i$, we obtain
\begin{eqnarray*}
  \tsr_0(t)&=&\sum_{i=0}^{\lfloor (p-1)/2\rfloor } (\alpha_i^2-\alpha_{p-i}^2) t^i \tsC_{p-2i-1} = \sum_{i=0}^{\lfloor (p-1)/2\rfloor } (\alpha_i^2-\alpha_{p-i}^2) \sum_{m=i}^{p-i-1} t^m.
\end{eqnarray*}
The equality \eqref{eq:res_0} is derived from this expression by regrouping the terms.
\end{proof}

\begin{remark}
It can be easily checked that the remainders, which have been calculated in the course of  proofs  of Theorems~\ref{th:AR1_3_1} and \ref{th:AR2_5_1}, can be deduced from Corollary~\ref{cor:tsr}.
For AR(1), we have $\tsr_0(t) = 1-\phi_1^2 = \alpha_0^2 - \alpha_1^2$.
For AR(2), we have $\tsr_1(t)=-\phi_1(1+\phi_2) = \alpha_0\alpha_1 - \alpha_1\alpha_2$ and
$\tsr_0(t) = (1-\phi_2^2)(t+1) = (\alpha_0^2-\alpha_2^2)(t+1)$.
\end{remark}

It follows from Corollary~\ref{cor:tsr} that one of the  necessary conditions of existence
of the  decomposition $\bfW=\bfA\ast \bfB$ for the case of a diagonal $\bfB$ with $B\ge p$
is $\phi_p=\pm 1$ (the coefficient in front of  $t^{p-1}$ in \eqref{eq:res_0} is equal to
$\alpha_0^2 - \alpha_p^2 = 1-\phi_p^2$ and should be equal to 0). This contradicts to $|\phi_p|<1$,  a necessary condition for the stationarity of AR($p$).
Therefore, for the matrix $\bfW$ corresponding to a stationary AR($p$) model, the  deconvolution \eqref{eq:b_decomp} cannot be performed.

\section{Conclusion}

We have considered the problem of matrix blind deconvolution; that is, finding matrices  $\bfA$ and $\bfB$ so that for a given matrix
 $\bfC$ we have $\bfC= \bfA \ast \bfB$. We have concentrated on the class of   matrices $\bfC$  proportional to the inverse autocovariance matrices of  autoregressive processes.
 We have shown  that the existence of the deconvolution of these matrices is equivalent to  the possibility of an equivalent  representation of a vector  HSLRA in the form of a matrix HSLRA and, as a consequence, for studying many versions of the so-called `Monte Carlo SSA', where the noise is red, and extensions of this method for more general autoregressive noise.

In the cases of  autoregressive models of orders one and two, we have fully characterized the range of parameters  where such deconvolution  can be performed
and provided construction schemes for performing deconvolutions. We have also considered stationary autoregressive models of order $p$, where we have proved that the
deconvolution $\bfC= \bfA \ast \bfB$ does not exist if the matrix $\bfB$ is diagonal and its size is larger than $p$.

We obtain an interesting and rather surprising theoretical fact: exact deconvolution of inverse autocovariance matrices corresponding to a stationary autoregressive noise can never be performed exactly.
There are several particular cases when such deconvolution is possible but in all these cases  the stationary conditions are violated.
Moreover, in the  most interesting special cases, where the deconvolution is possible, the values of coefficients of the autoregressive models lie on  the boundary of the
stationarity conditions.

Besides the theory, there are important implications of our results for applications of the HSLRA, SSA and related subspace-based methods for time series analysis. In particular, the results of this paper  explain why, in the case of stationary autoregressive noise, the problem of finding the optimal weights in the matrix form of the HSLRA can be solved only approximately.

There are many extensions of the HSLRA and SSA from time series to systems of time series, digital images and system identification problems; see  \cite{Markovsky2019,Markovsky.etal2006} for  HSLRA and \cite[Chapt.~4,5]{golyandina2018singular} for SSA. Hence the present  paper  can also be used as the basis for constructing  similar theoretical foundations for a  wider range of the problems.


\end{document}